\documentclass[a4paper,11pt]{amsart}

\usepackage[latin1]{inputenc}

\usepackage[T1]{fontenc}

\usepackage{lmodern}

\usepackage[english]{babel}
\usepackage{amsmath,amssymb,amsthm,accents}

\usepackage[pdftitle={Geometric partial differentiability on manifolds: the tangential derivative and the chain rule},pdfauthor={Alexandra Convent and Jean Van Schaftingen}]{hyperref}

\usepackage[abbrev]{amsrefs}
\usepackage[top=2.7cm,bottom=2.7cm]{geometry}
\usepackage{microtype}

\usepackage{enumerate}

\usepackage[babel=true]{csquotes}

\newtheorem{thm}{Theorem}[section]
\newtheorem{propo}[thm]{Proposition}
\newtheorem{lemme}[thm]{Lemma}
\newtheorem{cor}[thm]{Corollary}
\theoremstyle{definition}
\newtheorem{de}{Definition}[section]

\providecommand{\norm}[1]{\left\lVert#1\right\rVert}
\providecommand{\abs}[1]{\lvert#1\rvert}
\providecommand{\bignorm}[1]{\bigl\lVert#1\bigr\rVert}
\newcommand{\bigabs}[1]{\bigl\lvert #1 \bigr\rvert}

\newcommand{\R}{\mathbb{R}}
\newcommand{\N}{\mathbb{N}}
\usepackage[all]{xy}

\title[Geometric partial differentiability on manifolds]{Geometric partial differentiability on manifolds:\\ the tangential derivative and the chain rule}
\author{Alexandra Convent}
\address{
  Institut de Recherche en Math\'ematique et Physique\\
  Universit{\'e} catholique de Louvain\\
  Chemin du Cyclotron 2 bte L7.01.01, 1348 Louvain-la-Neuve, Belgium}
\email{Alexandra.Convent@uclouvain.be}

\author{Jean Van Schaftingen } 
\address{
  Institut de Recherche en Math\'ematique et Physique\\
  Universit{\'e} catholique de Louvain\\
  Chemin du Cyclotron 2 bte L7.01.01, 1348 Louvain-la-Neuve, Belgium}
\email{Jean.VanSchaftingen@uclouvain.be}



\begin{document}

\begin{abstract}
We define the tangential derivative, a notion of directional derivative which is invariant under diffeomorphisms. 
In particular this derivative is invariant under changes of chart and is thus well-defined for functions defined on a differentiable manifold. 
This notion is weaker than the classical directional derivative in general and equivalent to the latter for Lipschitz continuous functions.
We characterize also the pairs of tangentially differentiable functions for which the chain rule holds.
\end{abstract}

\maketitle

\section{Introduction}

Given a differentiable manifold \(M\), we are interested in defining a geometrically partial differentiation for functions from \(M\) to \(\R\).
That is, we would like to have a definition that implies classical directional differentiability when \(M=\R^m\) and that \emph{does not depend on local charts}. 
Equivalently we aim to define on \(\R^m\) a notion of partial differentiability which is invariant under diffeomorphisms. 

Let us first recall the classical definition of directional derivative which is not stable under maps \(T\colon \R^m\to \R^m\) such that \(T(V)\) is not locally in \(T(a)+ DT[V]\).
\begin{de}
Let \(A\subseteq \R^m\) and let \(V\subseteq \R^m\) be a linear subspace of \(\R^m\). 
A linear map \(L\in \mathcal{L}(V,\R^n)\) is a \emph{directional derivative} of a function \(f\colon A \to \R^n\) at \(a \in A \) with respect to \(V\) if 
for every \(\varepsilon >0\) there exists \(\delta >0\) such that if \(x\in A \cap (V+a)\) and \(\norm{x-a} \le \delta\), then
\[
\norm{f(x)- f(a) - L[x-a]} \le \varepsilon \norm{x-a}.
\]
\end{de}

The definition of directional derivative can be rephrased in terms of limits as 
\[
 \lim_{\substack{x \in V + a\\ x \to a}} \frac{\norm{f (x) - f (a) - L [x - a]}}{\norm{x - a}} = 0.
\]
 
The stronger notion that we are searching for is provided by the tangential derivative that we define as follows.
\begin{de}
\label{defTangential}
Let \(A\subseteq \R^m\) and let \(V\subseteq \R^m\) be a linear subspace of \(\R^m\). 
A linear map \(L\in \mathcal{L}(V,\R^n)\) is a \emph{tangential derivative} of a function \(f\colon A \to \R^n\) at \(a\in A\) with respect to \(V\) if 
for every \(\varepsilon >0\) there exist \(\delta > 0\) and \(\theta >0\) such that if \(x\in A\), \(v\in V\), \(\norm{x-a} \le \delta\) and \(\norm{x-a-v} \le \theta \norm{x-a}\), then 
\[
\norm{f(x) - f(a) - L[v]} \le \varepsilon \norm{x-a}.
\]
\end{de}

This new definition can be written in terms of limits as 
\[
 \lim_{\substack{v \in A\\ x \to a\\ \frac{\norm{x - a - v}}{\norm{x-a}} \to 0}} \frac{\norm{f (x) - f (a) - L [v]}}{\norm{x - a}} = 0.
\]

Geometrically, the tangential derivative considers the behavior of the function \(f\) in sharp cones around the plane \(V\). In particular, if the function \(f\) is tangentially differentiable, then the function \(f-f(a)\) grows linearly in a small sharp cone around the plane \(V\) : 
\begin{equation}\label{bounded_cone}
\limsup\limits_{\substack{x \to a \\ \frac{\mathrm{dist}(x-a,V)}{\norm{x-a}} \to 0}} \frac{\norm{f(x)-f(a)}}{\norm{x-a}} < \infty. 
\end{equation}
This condition \eqref{bounded_cone} is reminiscent of the directional Lipschitz condition \citelist{\cite{rock1979} \cite{rock1980}} but weaker. (If \(V=\R^m\), the latter condition is equivalent to the classical local Lipschitz condition.) 

We shall show that definition~\ref{defTangential} has the required invariance under diffeomorphisms (corollary~\ref{stableDiffeo}) and allows thus to define a tangential derivative on a differentiable manifold (definition~\ref{defManifold}). Although directional differentiability does not imply tangential differentiability in general, it is the case in particular for Lipschitz continuous functions (proposition~\ref{lipschitz}). 
This brings us to the original motivation of our work, which was to understand whether the directional derivative in the chain rule between Lipschitz continuous and bounded variation functions \cite{ambrosio} was well defined when working with maps between manifolds. 

We also show that the tangential derivative is the weakest notion of derivative which implies directional differentiability and which is invariant under diffeomorphisms. 
A serendipitous feature of this definition is that tangential differentiability with respect to \(V\) is equivalent to tangential differentiability with respect to the one-dimensional space spanned by \(v\) for every \(v\in V\) and linearity of the directional derivative (proposition \ref{characterization}).
This property is well-known to fail for directional derivatives. 

We finally investigate for the tangential derivative the chain rule which is a key property of derivatives (proposition~\ref{chainRule}). Although the rule \emph{does not hold in general}, it holds under a necessary and sufficient condition which is not very stringent, and it covers for example the case where the function on the left is Lipschitz continuous or where the function on the right has an injective tangential derivative.

\section{Examples and counterexamples of tangentially differentiable functions}

In this section, we give various examples and counterexamples of tangentially differentiable functions.

We first remark that in the limiting cases the definition covers known concepts. If \(V = \{0\}\), one observes by taking \(\theta = \frac{1}{2}\) that every function has a tangential derivative with respect to \(V\) and if \(V = \R^m\), the notions of directional and tangential derivatives are both equivalent with the classical definition of total differentiability of Stolz and Fr\'echet. 

In all the other cases, directional differentiability does not imply tangential differentiability. Given a linear subspace \(V\) such that \(\{0\} \subsetneq V \subsetneq \R^m\), we take \(\alpha >0\) and a linear map \(K\in \mathcal{L}(\R^m,\R^n)\) such that \(\mathrm{ker}(K) = V\), and we define the function \(f\colon \R^m \to \R^n\) for every point \(x\in \R^m\) by 
\begin{equation}\label{example}
f(x) = \left\{
\begin{aligned}
&\frac{K[x]}{\norm{x}^\alpha} && \text{ if } x\neq 0, \\
&0 && \text{ if } x = 0.
\end{aligned}\right.
\end{equation}
Since \(f=0\) on \(V\), the zero linear map is a directional derivative at the point \(0\) with respect to the subspace \(V\) but since \eqref{bounded_cone} is not satisfied, the function \(f\) does not have a tangential derivative at \(0\) with respect to the subspace \(V\).

It turns out however that for Lipschitz continuous function, the notions of directional and tangential derivatives are equivalent.
We recall that a function \(f\colon A\subseteq \R^m\to \R^n\) is \emph{Lipschitz continuous} if there exists a constant \(\kappa\ge0\) such that for all \(x,y\in A\), 
\(\norm{f(x)-f(y)} \le \kappa \norm{x-y}\).
We also denote by \(\langle v \rangle\) the \emph{vector space spanned} by \(v\), that is, \(\langle v \rangle = \{ r v \colon r \in \R\}\). 
\begin{propo}\label{lipschitz}
Let \(a \in A\subseteq \R^m\), let \(f\colon A \to \R^n\), let \(V\subseteq \R^m\) be a linear subspace of \(\R^m\) and let \(L\in \mathcal{L}(V,\R^n)\). 
If the function \(f\) is Lipschitz continuous, then the following statements are equivalent:
\begin{enumerate}[(i)]
\item \label{lip1} the linear map \(L\) is a directional derivative of \(f\) at \(a\) with respect to \(V\),
\item \label{lip2} the linear map \(L\) is a tangential derivative of \(f\) at \(a\) with respect to \(V\), 
\item \label{lip3} for every \(v\) in a dense subset of \(V\cap \partial B[0,1]\), the linear map \(L_{\arrowvert \langle v \rangle} \in \mathcal{L}(\langle v \rangle,\R^n)\) is a directional derivative of \(f\) at \(a\) with respect to \(\langle v \rangle\).
\end{enumerate} 
\end{propo}

\begin{proof}
Let \(\kappa\ge 0\) be given by the definition of Lipschitz continuity for the function \(f\). 
First we prove that \eqref{lip1} implies \eqref{lip2}. 
For every \(x\in A\) and every \(v\in V\), since \(f\) is Lipschitz continuous,
\begin{align*}
\norm{f(x)- f(a) - L[v]}
& \le \norm{f(x)-f(v+a)} + \norm{f(v+a) - f(a) - L[v]}\\
& \le \kappa\norm{x-a-v} + \norm{f(v+a) - f(a) - L[v]}.
\end{align*}
For every \(\varepsilon >0\), by definition of directional derivative, there exists \(\delta >0\) such that if \(y\in A \cap (V+a)\) and \(\norm{y-a} \le \delta\), then \(\norm{f(y)- f(a) - L[y-a]} \le \frac{\varepsilon}{4} \norm{y-a}\).
Therefore, if \(x\in A\), \(v\in V\), \(\norm{x-a}\le \frac{\delta}{2}\) and \(\norm{x-a-v} \le \min\bigl(1,\frac{\varepsilon}{2 \kappa}\bigr) \norm{x-a}\), then \(\norm{(v+a)-a} \le 2\norm{x-a} \le \delta\), and so \(\norm{f(x)- f(a) - L[v]} \le \varepsilon \norm{x-a}\).

Since it is obvious that \eqref{lip2} implies \eqref{lip3}, it remains to prove that \eqref{lip3} implies \eqref{lip1}.
For every \(x\in A \cap (V+a)\) and every \(v\in V\), 
\begin{align*}
& \norm{f(x)-f(a)-L[x-a]} \\
& \hspace{5mm} \le \norm{f(x)-f(\norm{x-a}v +a)- L[x-a-\norm{x-a}v]} \\
& \hspace{10mm} + \norm{f(\norm{x-a}v+a)-f(a)-L[\norm{x-a}v]} \\
& \hspace{5mm} \le (\kappa+\norm{L})\norm{x-a} \norm{ \frac{x-a}{\norm{x-a}} -v} +  \norm{f(\norm{x-a}v +a )-f(a)-L[\norm{x-a}v]}.
\end{align*}
By assumption, for every \(v\) in a dense subset \( S \subseteq V \cap \partial B[0,1]\), there exists \(\delta_v >0\) such that if \(y\in A \cap (\langle v \rangle +a)\) and \(\norm{y-a} \le \delta_v\), then \(\norm{f(y)- f(a) - L[y-a]} \le \frac{\varepsilon}{2} \norm{y-a}\). 
Here and in the sequel \(B[0,1]\) denotes the closed ball of radius \(1\) centered at the point \(0\).
Since the set \(V \cap \partial B[0,1]\) is compact and the set \(S\) is dense in \(V \cap \partial B[0,1]\), there exists \(N\in \N_*\) and \(v_1, \dots, v_N \in S\) such that 
\[
V \cap \partial B[0,1] \subseteq \bigcup_{i=1}^N B\left[v_i, \frac{\varepsilon}{2(1+\kappa + \norm{L})}\right].
\]
As a consequence, if \(x\in A \cap (V+a)\) and \(\norm{x-a}\le \min_{1\le i \le N} \delta_{v_i}\), then 
\[
\norm{f(x)-f(a)-L[x-a]} \le \varepsilon \norm{x-a}. \qedhere
\]
\end{proof}

The proof shows that the Lipschitz continuity assumption can be replaced by the weaker condition:
\[
\limsup_{\substack{\frac{\mathrm{dist}(x-a,V)}{\norm{x-a}} \to 0 \\ \frac{\mathrm{dist}(y-a,V)}{\norm{y-a}} \to 0}} \frac{\norm{f(x)-f(y)}}{\norm{x-y}} < \infty.
\]

\section{The chain rule for tangential derivatives}

We study now the chain rule for the tangential derivative.
We first remark that the chain rule \emph{does not} hold in general for the tangential derivative. For example, for every \(K\in \mathcal{L}(\R^m,\R)\) such that \(\{0\} \subsetneq \mathrm{ker}(K) \subsetneq \R^m\) and every \(\beta \ge 2\), we consider the functions \(f\colon \R^m \to \R\) and \(g \colon \R \to \R\) defined for every \(x\in \R^m\) and \(t\in \R\) by 
\begin{align*}
f(x) &= 
\begin{cases}
\frac{\abs{K[x]}^\beta}{\norm{x}} & \text{ if } x\neq 0,\\
0 & \text{ if } x = 0;
\end{cases}&
& \text{ and } &
g(t) &= \abs{t}^\frac{1}{\beta}.
\end{align*}
Since the zero linear map is a tangential derivative of the function \(f\) at \(0\) with respect to \(\mathrm{ker}(K)\), the function \(g\) has a tangential derivative at \(f(0)=0\) with respect to \(\{0\}\), but \(g \circ f\) has the same form as the function in \eqref{example} and so does not have a tangential derivative at \(0\) with respect to \(\mathrm{ker}(K)\). 

The chain rule holds however under an additional necessary and sufficient analytic assumption involving both functions in the composition.

\begin{propo}\label{chainRule}
Let \(A\subseteq \R^m\) and let \(V\subseteq \R^m\) be a linear subspace of \(\R^m\). 
Let \(L\in \mathcal{L}(V,\R^n)\) be a tangential derivative of \(f\colon A \to \R^n\) at \(a\in A\) with respect to \(V\) and let \(K\in \mathcal{L}(L[V], \R^p)\) be a tangential derivative of \(g\colon f(A) \to \R^p\) at \(f(a)\) with respect to \(L[V]\). 
The linear map \(K\circ L \in \mathcal{L}(V,\R^p)\) is a tangential derivative of \(g\circ f\) at \(a\) with respect to \(V\) if and only if 
\begin{equation}\label{CNS}
\lim_{\substack{v\in V \\ x \to a \\ \frac{\norm{x-a-v}}{\norm{x-a}} \to 0 \\ \frac{\norm{f(x)-f(a)}}{\norm{x-a}} \to 0 }} \frac{\norm{g(f(x))-g(f(a))}}{\norm{x-a}} =0.
\end{equation}
\end{propo}

\begin{proof}
We first assume that the function \(g\circ f\) is tangentially differentiable. For every \(x\in A\) and every \(v\in V\),
\begin{multline*}
\norm{g(f(x))-g(f(a))} 
\le \norm{(g\circ f)(x)-(g\circ f)(a) - (K\circ L)[v]} \\
 + \norm{K}(\norm{f(x)-f(a)-L[v]} + \norm{f(x)-f(a)}). 
\end{multline*}
We fix \(\varepsilon >0\). 
By assumption, there exist \(\delta, \theta >0\) such that if \(x\in A\), \(v\in V\), \(\norm{x-a}\le \delta\) and \(\norm{x-a-v}\le \theta \norm{x-a}\), then
\begin{gather*}
\norm{(g\circ f)(x)-(g\circ f)(a) - (K\circ L)[v]} \le \varepsilon \norm{x-a}
\intertext{and}
\norm{f(x)-f(a)-L[v]} \le \varepsilon \norm{x-a}.
\end{gather*}
As a consequence, if \(x\in A\), \(v\in V\), \(\norm{x-a}\le \delta\), \(\norm{x-a-v}\le \theta \norm{x-a}\) and \(\norm{f(x)-f(a)}\le \varepsilon\norm{x-a}\),
then
\[
\norm{g(f(x))-g(f(a))} \le (1+2\norm{K})\varepsilon \norm{x-a}, 
\]
that is, \eqref{CNS} is satisfied. 

Conversely, we assume that \eqref{CNS} holds. For every \(x\in A\) and every \(v\in V\),
\begin{equation}
\label{ayie}
\begin{aligned}
& \norm{(g\circ f)(x) - (g\circ f)(a) - (K\circ L)[v]} \\
& \hspace{10mm} \le \norm{g(f(x))-g(f(a))} + \norm{K}(\norm{f(x)-f(a)-L[v]} + \norm{f(x)-f(a)}). 
\end{aligned}
\end{equation}
We fix \(\varepsilon >0\). 
By assumption, there exist \(\delta_1, \theta_1 >0\) and \(\kappa >0\) such that if \(x\in A\), \(v\in V\), \(\norm{x-a}\le \delta_1\), \(\norm{x-a-v}\le \theta_1 \norm{x-a}\) and \(\norm{f(x)-f(a)}\le \kappa \norm{x-a}\), then  
\begin{align*}
 \norm{g(f(x))-g(f(a))} \le \varepsilon \norm{x-a} \text{ and }\norm{f(x)-f(a)-L[v]} \le \varepsilon \norm{x-a},
\end{align*}
and thus by \eqref{ayie},
\[
\norm{(g\circ f)(x) - (g\circ f)(a) - (K\circ L)[v]}  \le (\varepsilon (1+\norm{K}) + \kappa) \norm{x-a}.
\]
Without loss of generality, we assume that \(\kappa \le \varepsilon\), and thus if \(\norm{x-a}\le \delta_1\), \(\norm{x-a-v}\le \theta_1 \norm{x-a}\) and \(\norm{f(x)-f(a)}\le \kappa \norm{x-a}\), then 
\[
\norm{(g\circ f)(x) - (g\circ f)(a) - (K\circ L)[v]}  \le (2+\norm{K}) \varepsilon\norm{x-a}.
\]
It remains to cover the case where \(x\in A\) and \(\norm{f(x)-f(a)} > \kappa \norm{x-a}\). 
First, for every \(x\in A\) and every \(v\in V\) such that \(\norm{x-a} \le \delta_1\) and \(\norm{x-a-v}\le \theta_1 \norm{x-a}\), we have
\[
\norm{f(x)-f(a)} \le (\varepsilon + \norm{L}( 1+\theta_1))\norm{x-a}.
\] 
Let \(\eta = \varepsilon + \norm{L}( 1+\theta_1)\).
Since \(g\) is tangentially differentiable, there exist \(\Tilde{\delta}, \Tilde{\theta} >0\) such that if \(x\in A\), \(v\in V\), \(\norm{f(x)-f(a)}\le \Tilde{\delta}\) and \(\norm{f(x)-f(a)-L[v]}\le \Tilde{\theta} \norm{f(x)-f(a)}\), then
\[
\norm{g(f(x))-g(f(a)) - K[L[v]] } \le \frac{\varepsilon}{\eta} \norm{f(x)-f(a)}.
\]
Since \(L\) is a tangential derivative of \(f\), there exist \(\delta_2, \theta_2>0\) such that if \(x\in A\), \(v\in V\), \(\norm{x-a} \le \delta_2\) and \(\norm{x-a-v}\le \theta_2 \norm{x-a}\), then
\[
\norm{f(x) - f(a) - L[v]} \le \Tilde{\theta} \kappa \norm{x-a}.
\]
Therefore, if \(x\in A\), \(v\in V\), \(\norm{x-a}\le \min(\delta_1, \delta_2, \tfrac{\Tilde{\delta}}{\eta})\) and 
\[
\norm{x-a-v}\le \min(\theta_1,\theta_2)\norm{x-a},
\]
then
\(\norm{f(x)-f(a)} \le \Tilde{\delta}\) and \(\norm{f(x)-f(a)-L[v]} \le \Tilde{\theta}\norm{f(x)-f(a)}\), and so
\[
\norm{(g\circ f)(x) - (g\circ f)(a) - (K\circ L)[v]} \le \frac{\varepsilon}{\eta}\norm{f(x)-f(a)} \le \varepsilon \norm{x-a},
\]
that is, the function \(g\circ f\) is tangentially differentiable at \(a\).
\end{proof}

A first interesting consequence is a chain rule when the second function \(g\) in the composition is Lipschitz continuous. 
\begin{cor}
Let \(A\subseteq \R^m\) and let \(V\subseteq \R^m\) be a linear subspace of \(\R^m\). 
Let \(L\in \mathcal{L}(V,\R^n)\) be a tangential derivative of \(f\colon A \to \R^n\) at \(a\in A\) with respect to \(V\) and let \(K \in \mathcal{L}(L[V],\R^p)\) be a tangential derivative of \(g\colon f(A) \to \R^p\) at \(f(a)\) with respect to \(L[V]\). 
If the function \(g\) is Lipschitz continuous, then \(K \circ L\in \mathcal{L}(V,\R^p)\) is a tangential derivative of \(g\circ f\) at \(a\) with respect to \(V\). 
\end{cor}

\begin{proof}
If \(\kappa\ge 0\) is a Lipschitz constant for \(g\), condition \eqref{CNS} is satisfied since for every \(x\in A\setminus\{a\}\), 
\[
0 \le \frac{\norm{g(f(x))-g(f(a))}}{\norm{x-a}} 
= \frac{\norm{g(f(x))-g(f(a))}}{\norm{f(x)-f(a)}} \frac{\norm{f(x)-f(a)}}{\norm{x-a}} 
\le \kappa \frac{\norm{f(x)-f(a)}}{\norm{x-a}}. \qedhere
\]
\end{proof}

Another consequence which will be important in the sequel is the case where the tangential derivative of the first function \(f\) in the composition is injective. 

\begin{cor}\label{stableDiffeo}
Let \(A\subseteq \R^m\) and let \(V\subseteq \R^m\) be a linear subspace of \(\R^m\). 
Let \(L\in \mathcal{L}(V,\R^n)\) be a tangential derivative of \(f\colon A \to \R^n\) at \(a\in A\) with respect to \(V\) and let \(K\in \mathcal{L}(L[V], \R^p)\) be a tangential derivative of \(g\colon f(A) \to \R^p\) at \(f(a)\) with respect to \(L[V]\). 
If \(L\) is injective, then \(K \circ L\in \mathcal{L}(V,\R^p)\) is a tangential derivative of \(g\circ f\) at \(a\) with respect to \(V\).
\end{cor}

\begin{proof}
Since the linear map \(L\) is injective on its domain \(V\), there exists a constant \(\eta >0\) such that for every \(v\in V\), \(\norm{v} \le \eta \norm{L[v]}\). 
For every \(x\in A\) and every \(v\in V\), 
\begin{align*}
\norm{x-a} & \le \norm{x-a-v} + \eta \norm{L[v]} \\
&\le \norm{x-a-v} + \eta(\norm{f(x)-f(a)} + \norm{f(x)-f(a)-L[v]}). 
\end{align*}
Since \(L\) is a tangential derivative of \(f\), there exist \(\delta > 0\) and \(\theta>0\) such that if \(x\in A\), \(v\in V\), \(\norm{x-a} \le \delta\) and \(\norm{x-a-v} \le \theta\norm{x-a}\), then \(\norm{f(x)-f(a)-L[v]} \le \tfrac{1}{4\eta} \norm{x-a}\). 
Therefore, if \(\norm{x-a} \le \delta\) and \(\norm{x-a-v} \le \min\bigl(\theta,\tfrac{1}{4}\bigr)\norm{x-a}\), then \(\tfrac{3}{4}\norm{x-a} \le \eta \norm{f(x)-f(a)} + \norm{x-a-v}\), and so \(\norm{x-a} \le 2 \eta\norm{f(x)-f(a)}\). Hence, the condition \eqref{CNS} is trivially satisfied. 
\end{proof}

The previous chain rule shows the invariance under diffeomorphisms. As a consequence, we extend the notion of tangential derivative to functions on differentiable manifolds.
We recall that a \emph{local chart \(\varphi\)} of a differentiable manifold \(M\) of dimension \(m\) is a map \(\varphi \colon U \subseteq M \to \R^m\) such that \(U\) is an open subset of \(M\) and \(\varphi\colon U \to \varphi(U)\) is a diffeomorphism. For every \(p\in M\), we also denote by \(T_p M\) the \emph{tangent space at the point \(p\)}. 

\begin{de}\label{defManifold}
Let \(M\) be a differentiable manifold of dimension \(m\). Let \(p\in M\) and let \(V\) be a linear subspace of \(T_p M\). 
A linear map \(L \in \mathcal{L}(V,\R)\) is a \emph{tangential derivative} of a map \(f\colon M\to \R\) at \(p\) with respect to \(V\) if for some local chart \(\varphi \colon U\subseteq M \to \R^m\) such that \(p\in U\), \(L \circ D\varphi^{-1}(\varphi(p)) \in \mathcal{L}(D\varphi(p)[V], \R)\) is a tangential derivative of \(f\circ \varphi^{-1}\colon \varphi(U) \to \R^n\) at \(\varphi(p)\) with respect to \(D\varphi(p)[V]\). 
\end{de}
By corollary~\ref{stableDiffeo}, the previous definition is independent of the local chart \(\varphi\). 

\section{Characterizations of tangential differentiability}\label{chara}
In this section we characterize tangential differentiability as the weakest notion which implies directional differentiability and which is invariant under diffeomorphisms. It is also the weakest notion for which there is a chain rule with respect to nondegenerate paths with a linear dependence on the direction. Finally, tangential differentiability with respect to a subspace is equivalent with tangential differentiability with respect to every direction.
\begin{propo}\label{characterization}
Let \(a \in A\subseteq \R^m\), let \(f\colon A \to \R^n\), let \(V\subseteq \R^m\) be a linear subspace of \(\R^m\) and \(L\in \mathcal{L}(V,\R^n)\). 
The following statements are equivalent:
\begin{enumerate}[(i)]
\item \label{char1} the linear map \(L\) is a tangential derivative of \(f\) at \(a\) with respect to \(V\),
\item \label{char2} for each map \(\psi \colon U \subseteq \R^m \to \R^m\) such that \(U\subseteq A\) is open and \(\psi\colon U \to \psi(U)\) is a diffeomorphism, \(L \circ D\psi^{-1}(\psi(a)) \in \mathcal{L}(D\psi(a)[V], \R^n)\) is a directional derivative of \(f \circ \psi^{-1}\) at \(\psi(a)\) with respect to \(D\psi(a)[V]\),
\item \label{char3} for each path \(\gamma \in C^1((-1,1),\R^m)\) such that \(\gamma(0) = a\) and \(\gamma'(0)\in V\setminus \{0\}\),
the composite function \(f\circ \gamma\) is differentiable at \(0\) and \((f\circ \gamma)'(0) = L[\gamma'(0)]\),
\item \label{char4} for each \(v\in V \setminus \{0\}\), 
\begin{equation*}
\lim_{\stackrel{x\to a}{\frac{x-a}{\norm{x-a}} \to v }} \frac{f(x)-f(a) - \norm{x-a}L[v]}{\norm{x-a}} = 0.
\end{equation*}
\end{enumerate}
\end{propo}

If \(V= \R^m\), the equivalence between statements \eqref{char1}, \eqref{char3} and \eqref{char4} are well-known in finite dimensional theory. Indeed, statement \eqref{char3} is equivalent to Hadamard differentiability (\citelist{\cite{fichmann}*{teorema 1} \cite{penot73}\cite{shapiro}*{definition 2.2}}) and so it is equivalent --- since we are working in finite dimension --- to Fréchet--Stolz differentiability \eqref{char1}. Furthermore, statement \eqref{char4} is reminiscent to Bastiani--Michal differentiability \citelist{\cite{bastiani}*{d\'efinition 3.2}\cite{michal}\cite{penot}*{definition 1.1}} which is also equivalent to Hadamard differentiability. 

In comparison with Lipschitz continuous functions (proposition \ref{lipschitz}\eqref{lip1}), it is important in the previous statement \eqref{char4} to assume the property for \emph{every} \(v\in V\setminus\{0\}\) as it can be observed by taking \((v_n)_{n\in\N}\) a dense subset of \(\partial B[0,1] \subseteq \R^m\) and the function \(f \colon \R^m \to \R\) defined for every \(x\in \R^m\) by 
\[
f(x)=
\begin{cases}
0 & \text{ if } x \in \{ y \in \R^m \colon \mathrm{dist}(y,\langle v_n \rangle) \le \frac{1}{2^{n + 2}} \norm{y}\}, \\
1 & \text{ otherwise}.
\end{cases}
\]
Furthermore, property \eqref{char4} is well-known to fail for directional derivatives. 

Given a differentiable manifold \(M\) of dimension \(m\) and a function \(f \colon M \to \R\), the fact that \eqref{char4} implies \eqref{char1} gives us a sufficient condition for \(f\) to be tangentially differentiable at \(p\in M\) with respect to a linear subspace \(V \subseteq T_p M\). Indeed, if there exists \(L\in \mathcal{L}(V, \R)\) which is a tangential derivative of \(f\) at \(p\) with respect to \(\langle v \rangle\) for every \(v \in V\), that is, according to definition~\ref{defManifold}, for all local charts \(\varphi\colon U \subseteq M\to \R^m\) such that \(p\in U\) and for all \(w \in D \varphi(p)[V]\), 
\[
\frac{d}{dt} (f\circ\varphi^{-1})(\varphi(p) + t w)_{\arrowvert t=0} = \bigl(L\circ D\varphi^{-1}(\varphi(p))\bigr)[w],
\]
then \(f\) is tangentially differentiable at \(p\) with respect to \(V\). 

The proof of proposition~\ref{characterization} relies on the construction of an interpolation path between points of a sequence (for a similar construction, see \cite{delfour}*{theorem 2.6}).

\begin{lemme}\label{polynome}
Let \((x_n)_{n\in \N} \subseteq \R^m\) be a sequence that converges to \(a\in \R^m\). 
Let \((t_n)_{n\in \N} \subseteq (0,1)\) be a decreasing sequence such that \(\limsup_{n\to \infty} \tfrac{t_{n+1}}{t_n} < 1\). 
If the sequence \((\frac{x_n -a}{t_n})_{n\in \N}\) converges to \(v \in \R^m\), then there exists a path \(\gamma \in C^1([0,t_0],\R^m)\) such that \(\gamma(0) = a\), \(\gamma'(0) = v\) and for all \(n\in \N\), \(\gamma(t_n) = x_n\). 
\end{lemme}

\begin{proof}
We take \(p\) and \(q\) to be polynomials of degree \(3\) such that \(p(0)=0\), \(p(1)=1\) and \(p'(0)=p'(1)=0\), and \(q(0) = q(1)=0\) and \(q'(0)=q'(1) = 1\), that is, for every \(s\in \R\), \(p(s)=3s^2-2s^3\) and \(q(s) = s- p(s)\). We define the path \(\gamma \colon [0,t_0] \to \R^m\) for each \(t\in [0,t_0]\) by 
\[
\gamma(t) =
\begin{cases}
x_{n+1} + \left(x_n - x_{n+1}\right)p\left(\frac{t-t_{n+1}}{t_n -t_{n+1}}\right) + \left(t_n - t_{n+1}\right)q\left(\frac{t-t_{n+1}}{t_n -t_{n+1}}\right)v & \text{ if } t_{n+1} < t \le t_n, \\
a  & \text{ if } t = 0.
\end{cases}
\]
It is clear that \(\gamma\in C^1((0,t_0],\R^m)\). 
Since \(0 \le p\le 1\) on \([0,1]\) and \(\sup_{s\in[0,1]}\abs{q(s)}\le 2\), for every \(t\in (t_{n+1},t_n]\), 
\begin{equation*}
\begin{split}
\norm{\gamma(t)-\gamma(0)} 
& \le \norm{x_{n+1}-a} + \norm{x_n - x_{n+1}} \bigabs{p\bigl(\tfrac{t-t_{n+1}}{t_n -t_{n+1}}\bigr)} +   \abs{t_n-t_{n+1}}\bigabs{q\bigl(\tfrac{t-t_{n+1}}{t_n -t_{n+1}}\bigr)}\norm{v}\\
& \le \norm{x_{n+1}-a} + \norm{x_n - x_{n+1}} + 2 \abs{t_n-t_{n+1}}\norm{v}.
\end{split}
\end{equation*}
It follows that \(\gamma\) is continuous at \(0\). 
Next, since the sequence \((t_n)_{n\in \N}\) is decreasing and \(\limsup_{n\to \infty} \tfrac{t_{n+1}}{t_n}\) \( < 1\), there exists \(\alpha \in (0,1)\) such that for every \(n\in \N\), \(0 < t_{n+1} \le \alpha \, t_n\). 
For every \(t\in (t_{n+1},t_n]\), since \(p'+q' = 1\) on \([0,1]\), 
\begin{equation*}
\begin{split}
\norm{\gamma'(t) - v} & = \bigabs{p'\bigl(\tfrac{t-t_{n+1}}{t_n -t_{n+1}}\bigr)}\bignorm{\tfrac{x_n - x_{n+1}}{t_n-t_{n+1}} -v} \\
& \le \bigabs{p'\bigl(\tfrac{t-t_{n+1}}{t_n -t_{n+1}}\bigr)}\bignorm{\tfrac{x_n -a}{t_n} - v}+ \tfrac{t_n}{t_n-t_{n+1}} \bignorm{\tfrac{x_n-a}{t_n}- \tfrac{x_{n+1}-a}{t_{n+1}}} \\
& \le \tfrac{3}{2}\bignorm{\tfrac{x_n -a}{t_n} - v}+ \tfrac{1}{1-\alpha}\bignorm{\tfrac{x_n-a}{t_n}- \tfrac{x_{n+1}-a}{t_{n+1}}}.
\end{split}
\end{equation*}
As a consequence, \(\gamma'\) has a limit at \(0\) and so \(\gamma\) is differentiable at \(0\). We have thus proved that \(\gamma \in C^1([0,t_0],\R^m)\).
\end{proof}

\begin{proof}[Proof of proposition~\ref{characterization}]
Assume that \eqref{char1} holds. Since \(\psi\) is a local diffeomorphism, \(D\psi^{-1}(\psi(a)) \in \mathcal{L}(\R^m,\R^m)\) 
is a tangential derivative of \(\psi^{-1}\) at \(\psi(a)\) with respect to \(D\psi(a)[V]\) and is injective. By the chain rule (corollary~\ref{stableDiffeo}), since \(L\) is a tangential derivative of \(f\) at \(a = \psi^{-1}(\psi(a))\) with respect to \(V = D\psi^{-1}(\psi(a))[D\psi(a)[V]]\), \(L\circ D\psi^{-1}(\psi(a))\) is a tangential derivative of \(f\circ \psi^{-1}\) at \(\psi(a)\) with respect to \(D\psi(a)[V]\), and so a directional derivative. 

In order to prove that \eqref{char2} implies \eqref{char3}, since \(\gamma'(0) \in V\setminus \{0\}\), without loss of generality, if \(\gamma=(\gamma_1,\dots,\gamma_m)\), we assume that there exists \(\eta >0\) such that \(\gamma_1 \colon (-\eta,\eta)\to\R\) is one-to-one. 
We consider the function \(\Tilde{\psi} \colon (-\eta,\eta) \times \R^{m-1}\to \R^m\) defined for all \( (y_1,y'') \in (-\eta,\eta) \times \R^{m-1}\) as 
\[
\Tilde{\psi}(y_1,y'') = \gamma(y_1) + (0,y'').
\]
By construction, \( \Tilde{\psi} \colon (-\eta,\eta) \times \R^{m-1} \to \Tilde{\psi}((-\eta,\eta) \times \R^{m-1})\) is a diffeomorphism. 
So we take the function \(\psi = \Tilde{\psi}^{-1} \). 
Since \(f\circ \gamma = f\circ \psi^{-1}\) on \((-\eta, \eta)\times \{0\}\) and for every \(t\in (-\eta,\eta)\), \(D\psi^{-1}(\psi(a))[(t,0)] = \gamma'(0) t\in V\), \(f\circ \gamma\) is differentiable at \(0\) and \((f\circ \gamma)'(0) = L[\gamma'(0)]\).

We prove that \eqref{char3} implies \eqref{char4} by contradiction. 
Let \(v\in V \setminus \{0\}\) be such that \eqref{char4} is not satisfied. 
Then there exist \(\varepsilon >0\) and a sequence \((x_n)_{n\in \N}\subseteq A\) that converges to \(a\) such that \(\lim_{n \to \infty} \frac{x_n-a}{\norm{x_n-a}} =v\) but for every \(n\in \N\), 
\[
\norm{f(x_n)-f(a)- \norm{x_n-a} L[v]} > \varepsilon \norm{x_n-a}.
\]
Up to a subsequence, the sequence \((\norm{x_n-a})_{n\in \N}\) is decreasing and \(\limsup_{n\to \infty} \frac{\norm{x_{n+1}-a}}{\norm{x_n-a}} < 1\). 
Let \(\Tilde{\gamma}\in C^1([0,t_0],\R^m)\) be the path given by lemma~\ref{polynome} with \(t_n = \norm{x_n-a}\) for each \(n\in \N\).
Define \(\gamma \in C^1((-1,t_0],\R^m)\) as \(\gamma = \Tilde{\gamma}\) on \([0,t_0]\) and \(\gamma(t) = a + t v \) for all \(t\in (-1,0]\). 
As a consequence, \(\gamma'(0)=v \in V\) and
\begin{align*}
\norm{(f\circ \gamma)(t_n)-(f\circ \gamma)(0)- t_n L[\gamma'(0)]} 
& = \norm{f(x_n)-f(a)- \norm{x_n-a} L[v]} \\
& > \varepsilon \norm{x_n-a} = \varepsilon \abs{t_n},
\end{align*}
and so \(f\circ \gamma\) is not differentiable at \(0\).

Finally, we prove that \eqref{char4} implies \eqref{char1}.
We fix \(\varepsilon >0\). 
By assumption, for every \(v\in V\setminus\{0\}\), there exist \(\delta_v, \theta_v>0\) such that if \(x\in A\), \( \norm{x-a}\le \delta_v \) and \(\bignorm{\frac{x-a}{\norm{x-a}}- v} \le \theta_v\), then \(\bignorm{f(x)-f(a)-L[\norm{x-a} v ]} \le \frac{\varepsilon}{2} \norm{x-a}\).
Then we define \(\Tilde{\theta}_v = \min\bigl(\frac{\varepsilon}{2 \norm{L}},\frac{\theta_v}{2}\bigr)\). Since the set \(V \cap B[0,2]\) is compact, there exists \(N\in \N_*\) and \(v_1,\dots, v_N \in V \setminus\{0\}\) such that 
\[
V \cap B[0,2] \subseteq \bigcup_{i=1}^N B[v_i,\Tilde{\theta}_{v_i}].
\]
Finally, if \(x\in A\), \(v\in V\), \(\norm{x-a}\le \min_{1\le i \le N} \delta_{v_i}\) and
\[
\norm{x-a-v}\le \min\bigl(1,\min_{1\le i \le N}\frac{\theta_{v_i}}{2}\bigr) \norm{x-a},
\]
then \(\frac{\norm{v}}{\norm{x-a}} \le 2\) and so there exists \(v_i\in V\setminus\{0\}\) such that \(\norm{\frac{v}{\norm{x-a}}-v_i} \le \Tilde{\theta}_{v_i}\) and
\[ 
\norm{\frac{x-a}{\norm{x-a}} - v_i} \le \norm{\frac{x-a}{\norm{x-a}}-\frac{v}{\norm{x-a}}} + \norm{\frac{v}{\norm{x-a}}-v_i} \le \theta_{v_i},
\]
and thus 
\begin{align*}
\norm{f(x)-f(a)-L[v]} 
& \le \norm{f(x)-f(a)- \norm{x-a} L[v_i]}  + \norm{L}\norm{x-a}\norm{v_i-\frac{v}{\norm{x-a}}} \\
& \le \frac{\varepsilon}{2} \norm{x-a} + \norm{L} \norm{x-a} \frac{\varepsilon}{2(1+\norm{L})} \le \varepsilon \norm{x-a}. \qedhere
\end{align*}
\end{proof}


\end{document}